\documentclass[11pt]{amsart}

\usepackage{amsmath}
\usepackage{fullpage}
\usepackage{xspace}
\usepackage[utf8]{inputenc}
\usepackage{graphicx,color}
\usepackage[curve]{xypic} 
\usepackage{hyperref}
\usepackage{graphicx}

\usepackage{amsmath}%
\usepackage{amsthm}%
\usepackage{amscd}
\usepackage{amsfonts}%
\usepackage{amssymb}%
\usepackage{graphicx}

\usepackage{mathrsfs}

\usepackage{tikz}
\usetikzlibrary{matrix,arrows}
\newtheorem{theorem}{Theorem}[section]

\newtheorem{corollary}[theorem]{Corollary}

\newtheorem{lemma}[theorem]{Lemma}

\newtheorem{proposition}[theorem]{Proposition}

\theoremstyle{remark}
\newtheorem{remark}{Remark}[section]

\numberwithin{equation}{section}

\newcommand{\Mcal}{\mathscr{M}}

\newcommand{\Pro}{\mathbb{P}}

\newcommand{\Z}{\mathbb{Z}}
\newcommand{\C}{\mathbb{C}}

\newcommand{\Q}{\mathbb{Q}}
\newcommand{\R}{\mathbb{R}}

\newcommand{\A}{\mathbb{A}}

\newcommand{\End}{\mathrm{End}}

\begin{document}
\title[Bivariate polynomial injections]{Bivariate polynomial injections and elliptic curves}
\author{Hector Pasten}
\address{ Departamento de Matemáticas\newline
\indent Pontificia Universidad Católica de Chile\newline
\indent Facultad de Matemáticas\newline
\indent 4860 Av. Vicuña Mackenna\newline
\indent Macul, RM, Chile}
\email[H. Pasten]{hpasten@gmail.com}%
\thanks{This research was supported by FONDECYT Regular grant 1190442.}
\date{\today}

\begin{abstract} For every number field $k$, we construct an affine algebraic surface $X$ over $k$ with a Zariski dense set of $k$-rational points, and a regular function $f$ on $X$ inducing an injective map $X(k)\to k$ on $k$-rational points. In fact, given any elliptic curve $E$ of positive rank over $k$, we can take $X=V\times V$ with $V$ a suitable affine open set of $E$. The method of proof combines value distribution theory for complex holomorphic maps with results of Faltings on rational points in sub-varieties of abelian varieties.
\end{abstract}

\maketitle



\section{Introduction}

\subsection{Results} It is not known whether there is a bivariate polynomial $f\in\Q[x_1,x_2]$ inducing an injective function $\Q\times \Q\to \Q$. According to \emph{Remarque} $10$ in  \cite{Cornelissen} (which dates back to 1999) the existence of such an $f$ was first asked by Harvey Friedman, and Don Zagier suggested that $f=x_1^7+3x_2^7$ should have this property. In the direction of a positive answer to Friedman's question, we prove an injectivity result for bivariate polynomial functions on elliptic curves; in fact, our polynomial functions are very close to Zagier's suggestion.
\begin{theorem}\label{ThmIntro} Let $k$ be a number field. Let $a,b\in k$ with $4a^3+27b^2\ne 0$ and consider the smooth affine curve $C\subseteq \A^2_k$ defined over $k$ by the equation 
$$
y^2=x^3+ax+b.
$$ 
Let $\alpha,\beta, \gamma\in k$ with $\alpha,\beta\ne 0$ and $\gamma\notin \{-1,0,1\}$. Let $n\ge 9$ be a positive integer such that the only $n$-th root of unity in $k$ is $1$. Define the morphism $P:C\to \A^1_k$ by $P(x,y)=\alpha x+\beta y$ and define the morphism $f: C\times C\to \A^1_k$ by
$$
f((x_1,y_1),(x_2,y_2))=P(x_1,y_1)^n + \gamma \cdot P(x_2,y_2)^n=(\alpha x_1+\beta y_1)^n+\gamma (\alpha x_2+\beta y_2)^n.
$$ 
We have the following:
\begin{itemize}
\item[(i)] The map $C(k)\to k$ induced by $P$ on $k$-rational points is injective away from finitely many  points of $C(k)$.
\item[(ii)] Let $U\subseteq C$ be any non-empty Zariski open set defined over $k$ such that the function $U(k)\to k$ induced by $P$ is injective. Then the function $U(k)\times U(k)\to k$ induced by $f$ is injective away from finitely many $k$-rational points of $U(k)\times U(k)$. 
\end{itemize}
\end{theorem}
The previous theorem will be obtained in Section \ref{SecApplications} an application of  Theorem \ref{ThmMain} together with Theorem \ref{ThmUniqueness} and Proposition \ref{PropInj}.

A striking consequence of Theorem \ref{ThmIntro} is the following phenomenon: \emph{For every number field $k$ there is an affine algebraic variety $X$ of dimension greater than $1$ (in fact, a surface) with a morphism $f:X\to \A^1_k$ defined over $k$ such that the $k$-rational points of $X$ are dense in $X$, and nonetheless $f$ induces an injective function $X(k)\to k$ on $k$-rational points.}  More precisely, we have the following consequences of Theorem \ref{ThmIntro}. We refer to Section \ref{SecApplications} for details.

\begin{corollary}\label{CoroNF}  Let $k$ be a number field and let $E$ be an elliptic curve over $k$ of positive Mordell-Weil rank. There is a non-empty affine open set $V\subseteq E$ defined over $k$ and a morphism $f:V\times V\to \A^1_k$ also defined over $k$ such that the $k$-rational points of the affine surface $V\times V$ are  dense, and $f$ induces an injective map $V(k)\times V(k)\to k$ on $k$-rational points. 
\end{corollary}

\begin{corollary}\label{CoroQ} Let $C\subseteq \A^2_\Q$ be the affine curve defined over $\Q$ by $y^2=x^3+x-1$. Let $f$ be the polynomial function on $C\times C$ defined by
$$
f((x_1,y_1),(x_2,y_2)) = (x_1+y_1)^9 +2(x_2+y_2)^9.
$$
In the real topology, the $\Q$-rational points of the algebraic surface $C\times C$ are dense in $C(\R)\times C(\R)$, and the latter is homeomorphic to $\R^2$. 

Furthermore, the function $C(\Q)\times C(\Q)\to \Q$ induced by $f$ is injective away from finitely many $\Q$-rational points of $C\times C$.
\end{corollary}

Corollary \ref{CoroQ} shows, in particular, that there is no topological obstruction to a possible positive answer to Friedman's question.

An important tool in our approach is the concept of \emph{strong uniqueness function} for holomorphic maps from $\C$ to complex elliptic curves, which we introduce and study in Section \ref{SecUniqueness}. Then we apply them in Section \ref{SecArithmetic} to explicitly compute the Zariski closure of rational points in certain varieties associated to injectivity problems. 

Since we work with elliptic curves, the relevant varieties in our arguments appear as sub-varieties of abelian varieties. Results of Faltings allow us to study the rational points of these varieties \emph{provided} that we can explicitly compute the (translates of) positive dimensional abelian varieties contained in them. It is at this point where strong uniqueness functions are used, together with results of Hayman on holomorphic maps to Fermat varieties.


\subsection{Previous work} As far as we know, our results give the first unconditional progress on Friedman's question over $\Q$ and number fields. Besides this, there are two other results conditional on standard conjectures in Diophantine Geometry: 

In 1999,  Cornelissen \cite{Cornelissen} observed that the $4$-terms $abc$-conjecture \cite{BB} implies that polynomials of the form $x^n+3y^n$ induce injections $\Q\times \Q\to \Q$ if $n$ is odd and large enough. The basic idea is that failure of injectivity gives rise to a non-degenerate solution of the $4$-term Fermat equation $a^n+3b^n=c^n+3d^n$ where the $4$-terms $abc$-conjecture can be applied. Cornelissen also proved an unconditional analogue over function fields, where the $4$-terms $abc$-conjecture can be replaced by a theorem of Mason. 

In 2010, Poonen \cite{Poonen} proved that the Bombieri-Lang conjecture on rational points of surfaces of general type also implies a positive answer to Friedman's question.

Let us briefly recall the main ideas in Poonen's work \cite{Poonen}. He observed that when $f(x,y)\in\Q[x,y]$ is homogeneous,  it induces an injective map $\Q\times \Q\to \Q$ if and only if the $\Q$-rational points of the projective surface $Z=\{f(x_0,x_1)=f(x_2,x_3)\}\subseteq \Pro^3_\Q$ lie in the line $L=\{x_0=x_2, x_1=x_3\}\subseteq Z$. For general $f$ with $\deg(f)\ge 5$, the Bombieri-Lang conjecture for surfaces implies that the $\Q$-rational points of the surface $Z$ are algebraically degenerate, but this is not enough for injectivity as the Zariski closure of the rational points in $Z$ might contain other components besides $L$. While it is unclear how to exactly compute the Zariski closure of the rational points even under the Bombieri-Lang conjecture, Poonen managed to produce a suitably ramified cover of $Z$ to get rid of those other components (if any), leading to a new polynomial $f_0(x,y)\in \Q[x,y]$ of larger degree which has the desired injectivity property conditional on the Bombieri-Lang conjecture. 

Our approach of course owes some ideas to the work of Cornelissen (e.g. the use of $4$-terms Fermat equations as part of the construction) and Poonen (e.g. to formulate the injectivity problem as a problem about rational points in varieties). However, there are a number of crucial differences. For instance, in our context we need to consider higher dimensional varieties ---not just surfaces--- and we need to explicitly determine the Zariski closure of the relevant rational points. Our method for computing this Zariski closure pertains to the theory of value distribution of complex holomorphic maps.


\subsection{Additional motivation} Besides arithmetic interest, there are other motivations in the literature for studying bivariate polynomial injections.

In 1895, Cantor \cite{Cantor} showed that the polynomial
$$
f_1(x,y)=\frac{1}{2}(x+y)(x+y+1)+y
$$
defines a bijection $\Z_{\ge 0}\times \Z_{\ge 0}\to \Z_{\ge 0}$ and used this fact to prove that $\Z_{\ge 0}^2$ is countable. The polynomial $f_2(x,y)=f_1(y,x)$ has the same property, and its is an old open problem whether $f_1$ and $f_2$ are the only two polynomials inducing a bijection $\Z_{\ge 0}\times \Z_{\ge 0}\to \Z_{\ge 0}$. In 1923, Fueter and Polya \cite{FuPo} proved that this is the case among quadratic polynomials. In 1978, the result was extended to polynomials of degree at most $4$ by Lew and Rosenberg \cite{LeRo1, LeRo2}.

Bivariate injections $\Z\times \Z\to \Z$ are known as \emph{storing functions} in computability theory. Cornelissen \cite{Cornelissen} extended this notion to a model-theoretic context by introducing the concept of positive existential storing; his study of Friedman's question was developed in this setting.

More recently, the topic of bivariate polynomial injections has received attention for cryptographic purposes. Boneh and Corrigan-Gibbs \cite{BoCG} developed a new commitment scheme (among other applications) motivated by the conjecture that some bivariate polynomial $f\in\Q[x_1,x_2]$ such as Zagier's polynomial $x_1^7+3x_2^7$ defines an injective map $\Q^2\to \Q$. Attacks on this commitment scheme have been studied by Zhang and Wang \cite{ZhWa}.



\section{Strong uniqueness functions on elliptic curves} \label{SecUniqueness}

Let $\Mcal$ be the field of (possibly transcendental) complex meromorphic functions on $\C$. We recall that a polynomial $P(x)\in\C[x]$ is said to be a \emph{uniqueness polynomial} if the equation $P(f_1)=P(f_2)$ with $f_1,f_2\in\Mcal$ non-constant implies $f_1=f_2$. On the other hand, a polynomial $P(x)\in\C[x]$ is said to be a \emph{strong uniqueness polynomial} if the equation $P(f_1)=cP(f_2)$ with $f_1,f_2\in\Mcal$ non-constant and $c\in\C^\times$ implies $f_1=f_2$ (in particular, it implies $c=1$). Uniqueness polynomials and strong uniqueness polynomials are a classic topic in value distribution theory of holomorphic maps, and there is abundant literature on this subject, see for instance \cite{AnWaWo04, AvZa, Fujimoto, HuLiYa}.

Non-constant elements of $\Mcal$ are the same as holomorphic maps from $\C$ to $\Pro^1$, and by Picard's theorem the only other algebraic curves admitting non-constant holomorphic maps from $\C$ are elliptic curves. Let us introduce a notion of (strong) uniqueness function for the latter setting; such functions will naturally arise in our study of arithmetic injectivity problems.

Let $E$ be a complex elliptic curve. A rational function $P\in K(E)$ is a \emph{uniqueness function on $E$} if the equation $P(f_1)=P(f_2)$ with $f_1,f_2:\C\to E$ non-constant holomorphic maps implies $f_1=f_2$. On the other hand, a rational function $P\in K(E)$ is a \emph{strong uniqueness function on $E$} if the equation $P(f_1)=cP(f_2)$ with $f_1,f_2:\C\to E$ non-constant holomorphic maps and $c\in\C^\times$ implies  $f_1=f_2$ (in particular, it implies $c=1$).  It turns out that every complex elliptic curve admits strong uniqueness functions of a very simple kind.

\begin{theorem}\label{ThmUniqueness} Let $E$ be an elliptic curve over $\C$, choose a Weierstrass equation $y^2=x^3+Ax+B$ for $E$ and define the rational functions $x,y\in K(E)$ as corresponding coordinate projections. Let $\alpha,\beta\in \C$ be non-zero, and consider the non-constant rational function $P=\alpha x+\beta y\in K(E)$. Then $P$ is a strong uniqueness function on $E$.
\end{theorem}
\begin{remark} The conditions $\alpha\ne 0$ and $\beta\ne 0$ are necessary: For any complex elliptic curve $E$ and any non-constant holomorphic map $f_1: \C\to E$ we have
$$
x(f_1)=x([-1]\circ f_1) \quad \mbox{ and }\quad  y(f_1)=-1\cdot y([-1]\circ f_1).
$$
\end{remark}

\begin{remark} The conditions $\alpha\ne 0$ and $\beta\ne 0$ are necessary even if we work with uniqueness functions. This is clear for $\beta=0$. For $\alpha=0$ we can choose the elliptic curve $E$ with affine equation $y^2=x^3+1$. We consider the automorphism $u$ of $E$ defined by $(x,y)\mapsto (\epsilon x, y)$ with $\epsilon$ a primitive cubic root of $1$. Let us take any non-constant holomorphic map $f_1:\C\to E$ and define $f_2=u\circ f_1$. Then  $y(f_1)= y(u\circ f_1)$.
\end{remark}

\begin{remark} If $Q$ is a strong uniqueness polynomial for $\Mcal$ and $P$ is a uniqueness function on an elliptic curve $E$, then $Q\circ P$ is a strong uniqueness function on $E$. Thus, Theorem \ref{ThmUniqueness} together with standard results in the theory of strong uniqueness functions, allow one to construct strong uniqueness functions on $E$ of arbitrarily large degree.
\end{remark}

\begin{remark} After learning about this work, Michael Zieve managed to prove (private communication) that if $E$ is given in Weierstrass form as in Theorem \ref{ThmUniqueness} with coordinates $x$ and $y$, then ``most'' polynomials $P(x,y)\in \C[x,y]$ give strong uniqueness functions on $E$. Thus, strong uniqueness functions are abundant. The linear case given in Theorem \ref{ThmUniqueness} suffices for our purposes, although a more general description of strong uniqueness functions might be useful for other diophantine problems.
\end{remark}

In preparation for the proof of Theorem \ref{ThmUniqueness}, we record here the following simple lemma.
\begin{lemma}\label{LemmaHoloE} Let $E$ be an elliptic curve over $\C$.
\begin{itemize}
\item[(a)] Let $\theta :\C\to E$ be a non-constant holomorphic map. Then $\theta$ is surjective.
\item[(b)] Let $\psi:\C\to E\times E$ be a non-constant holomorphic map with algebraically degenerate image. Then the image of $\psi$ is exactly the translate of an elliptic curve subgroup of $E\times E$.
\end{itemize}
\end{lemma}
\begin{proof}  For (a), let us choose a uniformization $q:\C\to E$ and a holomorphic lift  $\tilde{\theta}:\C\to \C$ satisfying $q\circ \tilde{\theta} =\theta$. The entire function $\tilde{\theta}$ has at most one exceptional value in $\C$ by Picard's theorem, hence $\theta$ is surjective.

Item (b) is a special case of the literature around Bloch's conjecture, but this case admits a direct proof: The Zariski closure of the image of $\psi$ is an algebraic curve $Y$ in $E\times E$ of geometric genus $g=1$ ($g=0$ is excluded by the Riemann-Hurwitz theorem, and $g\ge 2$ is excluded by Picard's theorem). The only non-constant maps between elliptic curves are translates of isogenies, so, applying the two obvious projections $E\times E\to E$ we conclude that $Y$ is the translate of an elliptic curve subgroup of $E\times E$. By part (a), $\psi$ is surjective onto $Y$. 
\end{proof}
 For a complex elliptic curve $E$ we denote its neutral element by $0_E$.

\begin{proof}[Proof of Theorem \ref{ThmUniqueness}] 
Let $c\in\C^\times$ and let $f_1,f_2:\C\to E$ be non-constant holomorphic maps satisfying $P\circ f_1=c\cdot P\circ f_2$. Observe that both $P\circ f_1, P\circ f_2\in\Mcal$ are non-constant. We will prove that $f_1=f_2$, which implies $c=1$.

Let us define the holomorphic map 
$$
\phi: \C\to E\times E, \quad \phi(z)=(f_1(z),f_2(z)).
$$ 
We observe that $\phi$ does not have Zariski dense image in $E\times E$. This is because $P:E\to\Pro^1$ induces a finite morphism $P\times P: E\times E\to \Pro^1\times \Pro^1$, and the condition $P\circ f_1=c\cdot P\circ f_2$ implies that the image of $\phi$ is contained in the pull-back of  $\{([s:ct],[s:t]): [s:t]\in\Pro^1\}\subseteq \Pro^1\times \Pro^1$ under $P\times P$.

Therefore $\phi$ has algebraically degenerate image in $E\times E$. By part (b) of Lemma \ref{LemmaHoloE}, $\phi$ surjects onto an algebraic curve $Y\subseteq E\times E$ which is the translate of an elliptic curve subgroup in $E\times E$. 

We claim that that $(0_E,0_E)\in Y$, i.e. that $Y$ is in fact an elliptic curve subgroup of $E\times E$. This is because the only pole of $P=\alpha x+\beta y$ occurs at $0_E$, which is in the image of both $f_1$ and $f_2$ by part (a) of Lemma \ref{LemmaHoloE}. Since $P\circ f_1=c\cdot P\circ f_2$ with $c\in\C^\times$, it follows that 
$$
f_1^{-1}(0_E)=f_1^{-1}(P^{-1}(\infty))=(P\circ f_1)^{-1}(\infty)=(P\circ f_2)^{-1}(\infty)=f_2^{-1}(P^{-1}(\infty))=f_2^{-1}(0_E)
$$ 
which is non-empty. This proves $(0_E,0_E)\in Y$.

Therefore, the image of $\phi=(f_1,f_2):\C\to E\times E$ is an elliptic curve subgroup $Y\subseteq E\times E$, and by considering the projections $p_1,p_2:E\times E\to E$ we see that $Y$ is isogenous to $E$.

Fix an isogeny $v:E\to Y$ and let $u_j=p_j\circ v\in \End(E)$ for $j=1,2$. Since $v$ is \'etale, there is a holomorphic map $h:\C\to E$ that lifts $\phi$ via $v$ in the sense that $\phi=v\circ h$. Since $\phi=(f_1,f_2)$, we observe that $f_j=p_j\circ \phi=p_j\circ v\circ h=u_j\circ h$ for $j=1,2$.  It suffices to show $u_1=u_2$.

For a suitable lattice $\Lambda\subseteq \C$ (depending on our choice of Weierstrass equation for $E$), the associated Weierstrass function $\wp\in\Mcal$ induces an explicit uniformization of $w:\C\to E$, as it satisfies
$$
\left(\frac{1}{2} \wp'(z)\right)^2 = \wp(z)^3+A\wp(z)+B.
$$
The unformization $w:\C\to E$ is determined by $\wp$ in the sense that 
\begin{equation}\label{Eqwp}
x\circ w=\wp \quad \mbox{ and}\quad y\circ w=\frac{1}{2}\wp'.
\end{equation}
There are non-zero complex numbers $\lambda_1,\lambda_2\in\C^\times$ lifting the endomorphisms $u_1$ and $u_2$ via $w$, so that $(u_j\circ w)(z)=w(\lambda_j\cdot z)$. It suffices to prove $\lambda_1= \lambda_2$, as this will show $u_1=u_2$, hence, $f_1=f_2$.

Also, we recall that the Laurent expansions  of $\wp(z)$ and $\wp'(z)$ near $z=0$ are
$$
\wp(z)=\frac{1}{z^2} +\sum_{j\ge 1} c_j z^{2j}\quad \mbox{ and }\quad
\wp'(z)=\frac{-2}{z^3}  +\sum_{j\ge 1} 2jc_j z^{2j-1} 
$$
for suitable complex numbers $c_j\in\C$. Let us choose a lift $\tilde{h}:\C\to \C$ of the holomorphic map $h:\C\to E$ via the uniformization $w$, so that $h=w\circ \tilde{h}$. Furthermore, since $h:\C\to E$ is surjective (cf. item (a) in Lemma \ref{LemmaHoloE}) we can choose the holomorphic lift $\tilde{h}:\C\to \C$ in such a way that $0$ is in its image.

Recalling that $P=\alpha x+\beta y$ and that $f_j=u_j\circ h$, the condition $P\circ f_1=c\cdot P\circ f_2$ becomes
$$
\alpha\cdot  x\circ u_1\circ h + \beta\cdot  y\circ u_1\circ h = c\alpha\cdot  x\circ u_2\circ h + c\beta\cdot  y\circ u_2\circ h . 
$$
From the relation $(u_j\circ h)(z) = (u_j\circ w\circ \tilde{h})(z) = w(\lambda_j\cdot \tilde{h}(z))$ together with \eqref{Eqwp}, we deduce
$$
\alpha\cdot \wp (\lambda_1\cdot \tilde{h}(z)) + \frac{\beta}{2}\cdot \wp'(\lambda_1\cdot \tilde{h}(z))= c\alpha\cdot \wp (\lambda_2\cdot \tilde{h}(z)) + \frac{c\beta}{2}\cdot \wp'(\lambda_2\cdot \tilde{h}(z)).
$$
Since $0$ is in the image of $\tilde{h}:\C\to\C$, we have that the image of $\tilde{h}$ contains a neighborhood of $0\in\C$. From the previous relation we deduce that the following holds for the complex variable $z$ in a neighborhood of $0$ :
$$
\alpha\cdot \wp (\lambda_1\cdot z) + \frac{\beta}{2}\cdot \wp'(\lambda_1\cdot z)= c\alpha\cdot \wp (\lambda_2\cdot z) + \frac{c\beta}{2}\cdot \wp'(\lambda_2\cdot z).
$$
Considering the coefficients of $z^{-3}$ and $z^{-2}$ in the Laurent expansion near $z=0$, we deduce
$$
\beta \lambda_1^{-3} =c\beta \lambda_2^{-3}\quad \mbox{ and }\quad \alpha\lambda_1^{-2}=c\alpha\lambda_2^{-2}.
$$
Recall that $c,\alpha,\beta,\lambda_1,\lambda_2$ are non-zero complex numbers. We finally deduce 
$$
\lambda_2=\frac{\beta}{\alpha}\cdot\frac{c\alpha\lambda_2^{-2}}{c\beta\lambda_2^{-3}}=\frac{\beta}{\alpha}\cdot\frac{\alpha\lambda_1^{-2}}{\beta\lambda_1^{-3}}=\lambda_1.
$$
\end{proof}


\section{Arithmetic results} \label{SecArithmetic}

If $X$ and $Y$ are algebraic varieties over a field $k$ we write $X\times Y$ for $X\times_k Y$, and the diagonal in $X\times X$ is denoted by $\Delta_X$. 

The following proposition is an application of Faltings theorem for rational points on curves of genus at least $2$ and the notion of uniqueness function. The proof can be useful to illustrate the ideas in the proof of our main injectivity results for products of elliptic curves (cf. Theorem \ref{ThmMain}).
\begin{proposition} \label{PropInj} Let $k$ be a number field. Let $E$ be an elliptic curve over $k$ and let $P\in k(E)$ be a uniqueness function defined over $k$. Let $V\subseteq E$ be the locus where $P$ is regular. The function $V(k)\to k$ induced by $P$ is injective away from finitely many points of $V(k)$.
\end{proposition}
\begin{proof}
Using the morphism $P\times P: V^2\to \A^2_k$, we let $Z\subseteq V\times V$ be defined by $(P\times P)^*\Delta_{\A^1_k}$. Let $\overline{Z}\subseteq E^2$ be the Zariski closure of $Z$ and observe that $\overline{Z}-Z$ consists of finitely many points. We note that $\overline{Z}$ is a finite union of curves of geometric genus at least $1$, since it is a divisor on $E^2$, and the diagonal $\Delta_E$ is one of these curves.

We claim that the only irreducible component of $\overline{Z}$ of geometric genus $1$ is $\Delta_E$. In fact, let $Y$ be any irreducible component of $\overline{Z}$ with geometric genus $1$, then there is a non-constant holomorphic map $h:\C\to E^2$ whose image is contained in $Y$. We write $h=(h_1,h_2)$ with $h_j:\C\to E$ holomorphic and at least one of them non-constant. Let us observe that $P\circ h_1=P\circ h_2$ because $Y$ is a component of $\overline{Z}$ and $\overline{Z}-Z$ consists of finitely many points. It follows that both $h_1$ and $h_2$ are non-constant, and since $P$ is a uniqueness function we deduce that $h_1=h_2$. Hence $Y=\Delta_E$.

Finally, note that we must prove that only finitely many pairs $(x,y)\in (V^2-\Delta_E)(k)$ satisfy $P(x)=P(y)$, which is the same as proving that only finitely $k$-rational points of $Z$ lie outside $\Delta_E$. The result follows by Faltings theorem for curves \cite{Faltings1} and our previous genus computation.
\end{proof}
The following is our main tool for proving the injectivity results stated in the introduction.
\begin{theorem}\label{ThmMain} Let $k$ be a number field. Let $E$ be an elliptic curve over $k$ and let $P\in k(E)$ be a strong uniqueness function defined over $k$. Let $U\subseteq E$ be a non-empty Zariski open subset of $E$ defined over $k$ satisfying that $P$ is regular on $U$ and that the function $U(k)\to k$ induced by $P$ is injective. Let $n\ge 9$ be a positive integer such that the only $n$-th root of unity in $k$ is $1$. Let $\gamma\in k$ be different from $-1$, $0$, and $1$. Let $f:U\times U\to \A^1_k$ be the morphism defined over $k$ by
$$
f(q_1,q_2)= P(q_1)^n+\gamma\cdot  P(q_2)^n.
$$
The function $U(k)\times U(k)\to k$ induced by $f$ is injective away from finitely many $k$-rational points of $U\times U$.
\end{theorem}
\begin{remark} The function $P$ required by Theorem \ref{ThmMain} exists for every choice of $E$; our Theorem \ref{ThmUniqueness} gives an assortment of them. Furthermore, Proposition \ref{PropInj} shows that the open set $U$ required by Theorem \ref{ThmMain} always exists. Nevertheless, in concrete cases one might explicitly construct such an open set $U$ by other means such as a local argument, cf. the proof of Corollary \ref{CoroQ} in Section \ref{SecApplications} for instance.
\end{remark}
\begin{proof}[Proof of Theorem \ref{ThmMain}] Let $Z\subseteq U^4$ be the pull-back of $\Delta_{\A^1_k}$ by the morphism $f\times f: U^4\to \A^2_k$. Since $f$ is a finite map, $Z$ is a divisor in $U^4$ defined over $k$. Also, we have $\Delta_{U^2}\subseteq Z$. 

It suffices to show that all but finitely many $k$-rational points of $Z$ are contained in $\Delta_{U^2}$. 

Let $\overline{Z}$ be the Zariski closure of $Z$ in $E^4$. Let $D=\overline{Z}-Z$. Note that the Zariski closure of $\Delta_{U^2}\subseteq U^4 =U^2\times U^2$ in $E^4$ is precisely $\Delta_{E^2}$. It suffices to show that $\overline{Z}-(D\cup \Delta_{E^2})$ contains at most finitely many $k$-rational points.

By Faltings theorem on sub-varieties of abelian varieties \cite{Faltings2, Faltings3}, the (finitely many) irreducible components of the Zariski closure of $\overline{Z}(k)$ are translates of abelian sub-varieties of $E^4$ defined over $k$. Let $Y$ be any of these irreducible components with strictly positive dimension. It suffices to show that $Y\subseteq D\cup \Delta_{E^2}$, and we will prove this by contradiction.

For the sake of contradiction, let us suppose that $Y$ is not contained in $D\cup \Delta_{E^2}$. Since $Y$ is the translate of a positive-dimensional abelian sub-variety of $E^4$, there is a non-constant holomorphic map $h:\C\to Y$ with Zariski dense image (this can be seen by realizing $Y$ as a quotient of $\C^d$ by some lattice, and considering a holomorphic map to $\C^d$ with Zariski dense image). As $Y$ is not contained in $D\cup \Delta_{E^4}$, we can write $h=(h_0,h_1,h_2,h_3)$ with $h_j:\C\to E$ holomorphic for each $j$, satisfying:
\begin{itemize} 
\item[(i)] Not all the $h_j$ are constant. 
\item[(ii)] $(h_0,h_1)\ne (h_2,h_3)$ as holomorphic maps $\C\to E^2$, since $Y$ is not included in $\Delta_{E^2}$.
\item[(iii)] The compositions $P\circ h_j$ are well-defined elements of $\Mcal$, since $Y$ is not contained in $D$.
\item[(iv)] The following equation holds in $\Mcal$:
$$
(P\circ h_0)^n +\gamma\cdot (P\circ h_1)^n = (P\circ h_2)^n +\gamma \cdot (P\circ h_3)^n
$$
since $Y\subseteq \overline{Z}$ and $Z=(f\times f)^*\Delta_{\A^1_k}$.
\end{itemize}
By (i) and symmetry of the conditions (possibly replacing $\gamma$ by $1/\gamma$), we may assume that $h_0$ is non-constant. Hence $P\circ h_0$ is non-constant.

We claim that $P\circ h_2$ is not the zero constant. For otherwise, (iv) would give
\begin{equation}\label{EqFcurve}
(P\circ h_0)^n +\gamma\cdot (P\circ h_1)^n -\gamma \cdot (P\circ h_3)^n=0
\end{equation}
so that the holomorphic map $[P\circ h_0:P\circ h_1:P\circ h_3]:\C\to \Pro^2$ would have image contained in a Fermat curve of degree $n\ge 9$. Such a curve has genus at least $28$, so our holomorphic map is constant by Picard's theorem. This means that there are complex numbers $c_1,c_3\in \C$ with $P\circ h_1=c_1\cdot P\circ h_0$ and $P\circ h_3=c_3\cdot P\circ h_0$. Furthermore, not both $c_1$ and $c_3$ are equal to $0$, because of \eqref{EqFcurve} and the fact that $h_0$ is non-constant. We consider the two cases:
\begin{itemize}
\item If $c_1\ne 0$ then $h_1$ is non-constant. We get $h_1=h_0$ and $c_1=1$ because $P$ is a strong uniqueness function. From \eqref{EqFcurve} we get $(1+\gamma)(P\circ h_0)^n=\gamma \cdot (P\circ h_3)^n$ where $\gamma\ne 0$ and $1+\gamma\ne 0$. Since $P$ is a strong uniqueness function and $h_0$ is non-constant, we deduce $h_3=h_0$. This gives $1+\gamma=\gamma$, impossible.
\item If $c_3\ne 0$ then $h_3$ is non-constant. We get $h_3=h_0$ and $c_3=1$ because $P$ is a strong uniqueness function. From \eqref{EqFcurve} we get $(1-\gamma)(P\circ h_0)^n=-\gamma \cdot (P\circ h_1)^n$ and similarly we deduce $h_1=h_0$. This gives $1-\gamma=-\gamma$, impossible.
\end{itemize}
This contradiction proves that $P\circ h_2$ is not the zero constant.

We now claim that there is no $c\in \C^\times$ for which $P\circ h_2=c\cdot P\circ h_0$. For the sake of contradiction, suppose there is such a $c$. Then $h_2$ is non-constant because $h_0$ is non-constant, and since $P$ is a strong uniqueness function we deduce $h_0=h_2$ and $c=1$. Therefore $P\circ h_2= P\circ h_0$. By item (iv) and the condition $\gamma\ne 0$ we deduce $(P\circ h_1)^n =(P\circ h_3)^n$.  If $h_1$ or $h_3$ is non-constant, so is the other and we deduce that they are equal because $P$ is a strong uniqueness function; this would contradict item (ii) because we already know $h_0=h_2$. Therefore both $h_1$ and $h_3$ are constant. If $h_1,h_3$ are the same constant function, then we again get a contradiction with (ii), so they are different constant functions given by two different points $q_1,q_3\in E$ respectively. We observe that $q_1,q_3\in U(k)$ because $Y(k)$ is Zariski dense in $Y$, which is not contained in $D$. Since $P$ is injective on $U(k)$, our hypothesis on $n$ (the only $n$-th root of unity in $k$ is $1$) shows that $(P\circ h_1)^n=P(q_1)^n\ne P(q_3)^n=(P\circ h_3)^n$, and we obtain a contradiction. This proves that there is no $c\in \C^\times$ for which $P\circ h_2=c\cdot P\circ h_0$.

Let us define the holomorphic map $H:\C\to \Pro^3$ by 
$$
H= [P\circ h_0: P\circ h_1: P\circ h_2:P\circ h_3]
$$ 
and observe that it is non-constant because $P\circ h_0$ is non-constant while $P\circ h_2$ is not of the form $c\cdot P\circ h_0$ for any $c\in \C$ (the cases $c=0$ and $c\ne 0$ are covered by the two previous claims). Furthermore, we observe that the image of $H$ is contained in the Fermat surface $F\subseteq \Pro^3$ defined by
\begin{equation}\label{EqFermat}
x_0^n+\gamma\cdot x_1^n-x_2^n-\gamma\cdot x_3^n=0.
\end{equation}
We recall that $\gamma\ne 0$ and $n\ge 9$. By Hayman's theorem \cite{Hayman} the image of $H$ must be contained in one of the ``obvious'' lines of $F$ determined by a vanishing $2$-terms sub-sum of \eqref{EqFermat}. Since $P\circ h_2$ is not of the form $c\cdot P\circ h_0$ for any $c\in \C$, this only leaves the following possibilities:
\begin{itemize}
\item $P\circ h_0 = \lambda \cdot P\circ h_1$ with $\lambda^n=-\gamma$. This gives that $h_1$ is non-constant, and since $P$ is a strong uniqueness function we would get $h_1=h_0$ and $\lambda=1$. This is not possible since $\gamma\ne -1$.
\item $P\circ h_0 = \lambda \cdot P\circ h_3$ with $\lambda^n=\gamma$. This gives that $h_3$ is non-constant, and since $P$ is a strong uniqueness function we would get $h_3=h_0$ and $\lambda=1$. This is not possible since $\gamma\ne 1$.
\end{itemize}
This is a contradiction. Therefore $Y$ must be contained in $D\cup \Delta_{E^2}$.
\end{proof}



\section{Applications}\label{SecApplications}

We can now use the results of Sections \ref{SecUniqueness} and \ref{SecArithmetic} to prove the results stated in the introduction.

\begin{proof}[Proof of Theorem \ref{ThmIntro}] Let $E$  be the projective closure of $C$, then $E$ is an elliptic curve over $k$. Note that $P\in k(E)$ is a strong uniqueness function defined over $k$, by Theorem \ref{ThmUniqueness}. Furthermore, $P$ is regular on $C\subseteq E$, and we see that item (i) of Theorem \ref{ThmIntro} follows from Proposition \ref{PropInj}.

Item (ii) of Theorem \ref{ThmIntro} follows from Theorem \ref{ThmMain}.
\end{proof}

\begin{proof}[Proof of Corollary \ref{CoroNF}] We choose a short Weierstrass equation for $E$ and then we apply Theorem \ref{ThmIntro} to the corresponding affine curve $C\subseteq E$. By item (i) of Theorem \ref{ThmIntro} we can shrink $C$ to get an open set $U$ as required by item (ii). Deleting finitely many $k$-rational points from $U$ we get the desired open set $V$.
\end{proof}

\begin{proof}[Proof of Corollary \ref{CoroQ}] The elliptic curve $E$ of affine equation $y^2=x^3+x-1$ has Cremona label 248c1. The group $E(\Q)$ is isomorphic to $\Z$, generated by the point $(1,1)$. In the real topology, we note that $E(\R)$ is connected, hence $E(\Q)$ is dense in $E(\R)$  (its closure in the one-dimensional real Lie group $E(\R)$ is open). 

Consider the morphism $P:C\to \A^1_\Q$ given by $P(x,y)=x+y$. We claim that $P$ induces an injective map on real points $C(\R)\to\R$. If not, we would have two different points $q_1,q_2\in C(\R)$ such that the line through them has slope $-1$. Hence, for some $q\in C(\R)$ (between $q_1$ and $q_2$) the line tangent to $C(\R)$ at $q$ has slope $-1$. However, it is a simple computation to check that the slope of any non-vertical tangent of $C(\R)$ has absolute value larger than $2.708$. In particular, no tangent to $C(\R)$ can have slope $-1$, which proves injectivity on real points.

 Since $P$ induces an injective function $C(\R)\to \R$, it also induces an injective function $C(\Q)\to \Q$. Thus, we can apply item (ii) of Theorem \ref{ThmIntro} with $U=C$.
\end{proof}


\section{Acknowledgments}

This research was supported by FONDECYT Regular grant 1190442.

The results in this work were motivated by discussions at the PUC number theory seminar and I deeply thank the attendants. I also thank Gunther Cornelissen and Michael Zieve for comments on an earlier version of this work.


\end{document}